\documentclass[11pt]{amsart}
\title{On the Gromov--Witten/Donaldson--Thomas Correspondence and Ruan's Conjecture for Calabi-Yau $3$-Orbifolds}
\author{Dustin Ross}
\address{Dustin Ross, University of Michigan, Department of Mathematics, Ann Arbor, MI 48109, USA}
\email{dustyr@umich.edu}

\usepackage {color,graphicx,psfrag, verbatim,amssymb,amscd,enumerate,subfigure,ytableau,tikz}
\usepackage{texdraw,xypic,appendix}

\newcommand{\rk}{\textrm{rk}}

\newcommand{\bP}{\mathbb{P}}

\newcommand{\bZ}{\mathbb{Z}}
\newcommand{\aut}{\mathrm{Aut}}

\newcommand{\cO}{\mathcal{O}}
\newcommand{\bQ}{\mathbb{Q}}
\newcommand{\bC}{\mathbb{C}}

\newcommand{\cB}{\mathcal{B}}

\newcommand{\fq}{\mathfrak{q}}
\newcommand{\bq}{\mathbf{q}}
\newcommand{\cX}{\mathcal{X}}
\newcommand{\cN}{\mathcal{N}}

\newcommand{\cZ}{\mathcal{Z}}

\newcommand{\Mbar}{\overline{\mathcal{M}}}
\newcommand{\bE}{\mathbb{E}}
\newcommand{\cI}{\mathcal{I}}

\newtheorem*{crc}{Theorem 2: Crepant Resolution Conjecture}
\newtheorem*{gwdt}{Theorem 1: Orbifold GW/DT Correspondence}

\newtheorem{dummy}{}[section]
\newtheorem{lemma}[dummy]{Lemma}

\newtheorem{theorem}[dummy]{Theorem}

\newtheorem{Theorem}{Theorem}
\newtheorem*{corollary*}{Corollary 2}

\theoremstyle{definition}

\newtheorem{definition}[dummy]{Definition}

\newtheorem{remark}[dummy]{Remark}

\usepackage{graphicx}

2
\begin{document}


\begin{abstract}
For any toric Calabi-Yau $3$-orbifold with transverse $A$-singularities, we prove the Gromov--Witten/Donaldson--Thomas correspondence and Ruan's crepant resolution conjecture in all genera.
\end{abstract}

\maketitle

\section{Introduction}

\subsection{Summary of Results}

This paper addresses two fundamental questions in Gromov--Witten (GW) theory:
\begin{enumerate}
\item Ruan's crepant resolution conjecture (CRC) which states that the GW theory of an orbifold should be related to that of a crepant resolution, and
\item the Gromov--Witten/Donaldson--Thomas correspondence (GW/DT) of Maulik--Nekrasov--Okounkov--Pandharipande \cite{mnop:gwdti} which states that the GW theory of a Calabi-Yau (CY) threefold should be related to the Donaldson--Thomas theory of that same threefold.
\end{enumerate}

More specifically, for $\pi:W\rightarrow\cZ$ a crepant resolution of a hard-Lefschetz orbifold, Bryan--Cadman--Young \cite{bcy:otv} conjectured the following square of equivalences:
\[
\begin{CD}
GW(W) @= DT(W)\\
@| @|\\
GW(\cZ) @= DT(\cZ)
\end{CD}
\]
where all equalities consist of an identification of variables in the formal series, followed by analytic continuation.  This paper completes our understanding of the square of equivalences when $\cZ$ is a toric CY $3$-orbifold with transverse $A$-singularities (ie. the orbifold structure is cyclic and supported on disjoint torus lines).

Prior to this work, for $\cZ$ a toric CY $3$-orbifold with transverse $A$-singularities, the top equality was a theorem of Maulik--Oblomkov--Okounkov--Pandharipande \cite{moop:gwdtc} while the right equality was a theorem of the author \cite{r:dtcrc}. Partial results concerning the left and bottom equalities were obtained by Coates--Corti--Iritani--Tseng, Zhou, Brini--Cavalieri--Ross \cite{ccit:cgztgwi,z:crcag,bcr:craos} and Ross--Zong, Zong \cite{rz:ggmv,rz:chialsf,z:gmvf}, respectively. Building upon these partial results we prove the following (see Section \ref{sec:results} for precise statements).

\begin{gwdt}
Let $\cZ$ be a toric CY $3$-orbifold with transverse $A$-singularities. Let $GW(\cZ)$ denote the all genus, primary GW potential of $\cZ$ and let $DT(\cZ)$ denote the reduced, multi-regular DT potential of $\cZ$. Then there is an explicit identification of formal parameters so that
\[
GW(\cZ)=DT(\cZ)
\]
\end{gwdt}

Theorem \ref{thm:gwdt} completes the bottom equality in the square. Composing the bottom, right, and top equalities we deduce the following.

\begin{crc}
Let $\pi:W\rightarrow\cZ$ be the toric crepant resolution of $\cZ$.Then there is an explicit affine-linear identification of formal parameters so that \[GW(W)=GW(\cZ)\]
\end{crc}

\begin{remark}
There is no need for analytic continuation in either of our theorems because the generating series can all be expressed as rational functions in an appropriate set of variables.
\end{remark}

\subsection{Outline of Proof}

Our methods employ the philosophy of the \emph{topological vertex} \cite{akmv:tv}. Essentially, the topological vertex asserts that the GW or DT theory of a toric CY $3$-fold can be reduced to the study of a local generating series defined at each torus fixed point. In GW theory, the local theory is a generating series of (orbifold) Hodge integrals (or \emph{open} GW invariants). In DT theory, the local theory is a generating series of piles of (colored) boxes. 

In \cite{rz:chialsf}, it was shown that Theorem \ref{thm:gwdt} could be reduced to a local correspondence on the level of the topological vertex. The main content of this paper is a proof of that local statement (Theorem \ref{thm:vertexcorrespondence}). The proof of Theorem \ref{thm:vertexcorrespondence} is rather roundabout, and relies heavily on previous work. Essentially, Theorem \ref{thm:vertexcorrespondence} is the composition of three previous results: the open CRC of Brini--Cavalieri--Ross \cite{bcr:craos}, the GW/DT correspondence of Maulik--Oblomkov--Okounkov--Pandharipande \cite{moop:gwdtc}, and the DT vertex CRC of the author \cite{r:dtcrc}.

Once Theorem \ref{thm:gwdt} is proved, Theorem \ref{thm:crc} follows quickly.

\subsection{Plan of Paper}

In Section \ref{sec:results}, we define the relevant GW and DT potentials and we give precise statements of Theorems \ref{thm:gwdt} and \ref{thm:crc}. In Section \ref{sec:local} we introduce the orbifold vertex both in GW and DT theory. These are generating series associated to each torus fixed point in $\cZ$. In Section \ref{sec:vertexgwdt}, we state and prove the GW/DT correspondence on the level of the orbifold vertex (Theorem \ref{thm:vertexcorrespondence}).

\subsection{Acknowledgements} 

The author is greatly indebted to his collaborators Andrea Brini, Renzo Cavalieri, and Zhengyu Zong. The current work relies heavily on their previous work together and he is extremely thankful for their encouragement in this project. The author has been supported by NSF RTG grants DMS-0943832 and DMS-1045119 and the NSF postdoctoral research fellowship DMS-1401873.

\section{Main Results}\label{sec:results}

In this section we give a precise statement of the main results. We begin by setting up some geometric notation.

\subsection{Generators for Cohomology}\label{sec:bases}

Let $\cZ$ be a toric CY $3$-orbifold with transverse $A$-singularities, ie. $\cZ$ is a smooth, quasi-projective, toric Deligne--Mumford stack of dimension three over $\bC$ with trivial canonical bundle and with cyclic isotropy supported on disjoint lines. Denote by $\{L_i\}$ the (not necessarily compact) singular lines in $\cZ$ where $L_i$ has isotropy group $\bZ_{n_i}$. Fix an isomorphism of $\bZ_{n_i}\cong\left\langle \xi_{n_i}:=e^{2\pi\sqrt{-1}/n_i}\right\rangle$. When $L_i$ is compact, also fix an integer $m_i$ so that  $\mathcal{N}_{L_i/\cZ}=\cO(m_i)\oplus\cO(-m_i-2)$ and the isotropy group acts with weight one on the fibers of the first factor. 

We will be interested in the primary GW invariants of $\cZ$. Because $\cZ$ is a CY $3$-fold, the string equation implies that the only nonvanishing primary invariants are those for which all cohomological insertions have Chen-Ruan degree $2$. For each \emph{compact} line $L_i$, let $A_i$ be the dual divisor and let $\{B_i\}$ be a set of divisors disjoint from the $L_i$ such that $\{A_i,B_i\}$ is a generating set for $H^2(\cZ,\bQ)$. Let $\phi_i^j$ denote the fundamental class on the $j$th twisted sector in the orbifold cohomology of $L_i$. Then the set $\{A_i,B_i,\phi_i^j\}$ generates $H^2_{CR}(\cZ,\bQ)$.

Now let $\pi:W\rightarrow\cZ$ be the toric crepant resolution of $\cZ$. Then $\pi$ is an isomorphism away from the $L_i$. If $L_i$ is compact, its preimage under $\pi$ consists of $n_i-1$ Hirzebruch surfaces 
\[
H_{i,k}:=\bP(\cO_{\bP^1}\oplus\cO_{\bP^1}(n_im_i+2(n_i-k))) \hspace{1cm} 1\leq k< n_i
\]
where the $\infty$-section of $H_{i,k}$ is glued to the $0$-section of $H_{i,k+1}$. Let $C_i$ be the divisor dual to the $0$ section of $H_{i,1}$ and let $D_{i,k}$ be the divisor dual to the fiber of $H_{i,k}$. If $L_i$ is not compact, then its preimage is a chain of $\bP^1$s with normal bundles $\cO_{\bP^1}\oplus\cO_{\bP^1}(-2)$, in this case we still have divisors $D_{i,k}$ dual to the $\bP^1$s. The set $\{B_i,C_i,D_{i,k}\}$ generates $H^2(W)$ where we have identified $B_i$ with its image under $\pi^*$.

\subsection{Gromov-Witten Theory}

GW invariants are virtual intersection numbers on $\Mbar_{g,k}(\cZ,\beta)$, the Kontsevich moduli stack of genus $g$, $k$ pointed stable maps to $\cZ$ of degree $\beta\in H_2(\cZ,\bZ)$.  Let $\gamma=\{\gamma_1,\dots,\gamma_k\}$ denote a multiset of elements in $\{\phi_i^j\}$ with $m_{i,j}$ entries equal to $\phi_i^j$ and let ${\bf t}$ be a formal parameter in $H^2(\cZ)$. Let $ev_i:\Mbar_{g,k}(\cZ,\beta)\rightarrow\cI\cZ$ denote the evaluation map to the inertia stack.  If $\Mbar_{g,k}(\cZ,\beta)$ is projective, the GW partition function is defined by
\[
GW_\cZ^{\bullet}({\bf x},{\bf t},u,v):=\exp\left(GW_\cZ({\bf x},{\bf t},u,v)\right)
\]
where
\begin{equation}\label{eq:intie}
GW_\cZ({\bf x},{\bf t},u,v):=\sum_{\substack{\beta\neq 0 \\ g,\gamma,n}}\left(\int_{[\Mbar_{g,k+n}(\cX,\beta)]^{vir}}ev^*(\gamma)\frac{ev^*({\bf t})^n}{n!}\right)\frac{{\bf x}^\gamma}{\gamma!}u^{2g-2} v^\beta
\end{equation}
with
\[
ev^*(\gamma)=\prod_{i}ev_{i}^*(\gamma_i)
\]
and
\[
\frac{{\bf x}^\gamma}{\gamma!}:=\prod_{l,k}\frac{x_{i,j}^{m_{i,j}(\gamma)}}{m_{i,j}(\gamma)!}.
\]

We interpret $({\bf x},{\bf t})$ as a formal parameter on $H^2_{CR}(\cZ)$.

By the divisor equation, the dependence on ${\bf t}$ and $v$ is redundant. In particular, they always appear together as a factor $v^\beta e^{(\bf{t},\beta)}$ where $(-,-)$ is the intersection pairing. For this reason, we freely omit the ${\bf t}$-dependence in our discussion of the GW/DT correspondence while we omit the $v$-dependence in the discussion of the CRC.

In the case of interest to us, $\Mbar_{g,k}(\cZ,\beta)$ may not be projective because $\cZ$ is only quasi-projective. In the non-projective case, the GW partition function can still be defined as follows. Consider the full torus $(\bC^*)^3$ acting on $\cZ$. Inside of it, there is a two-dimensional torus $T\cong(\bC^*)^2$ which acts trivially on the canonical bundle. We can lift the $T$-action to $\Mbar_{g,k}(\cZ,\beta)$ and compute
\begin{equation}\label{eq:fixie}
\sum_{F}\int_{[F]^{vir}}\frac{\iota_F^*\left(ev^*(\zeta)ev^*({\bf t})^n\right)}{e^{eq}\left(\cN_F\right)}\in \bQ(u_1,u_2)
\end{equation}
where the sum is over the fixed loci $\iota_F:F\hookrightarrow \Mbar_{g,k+n}(\cZ,\beta)^{T}$, the denominator is the equivariant Euler class of the normal bundle, and $H^*(\cB T)=\bQ[u_1,u_2]$. Choosing a subtorus $\bC^*\subset T$, we can restrict \eqref{eq:fixie} to $H^*(\cB\bC^*)=\bQ[u]$. Since the insertions all have degree two, the end result does not depend on $u$, i.e. it is a rational number. To define the GW partition function in the non-projective case, we replace the integral in \eqref{eq:intie} with this rational number. A priori, one would expect these invariants to depend on the choice of subtorus $\bC^*\subset T$. However, we will see below that the GW invariants are actually independent of this choice.

\subsection{Donaldson-Thomas Theory}

DT invariants are intersection numbers on the Hilbert scheme of substacks in $\cZ$.  Hilbert schemes are indexed by the (compactly supported) $K$ group of coherent sheaves.  Let $\cO_{i,j}$ denote the skyscraper sheaf supported on a generic point of the orbifold line $L_i$ for which $\bZ_{n_i}$ acts by multiplication by $\exp\left( \frac{2\pi\sqrt{-1}}{n_i}j \right)$.  For $\gamma$ and $\beta$ as above, let $Hilb_{\gamma}(\cZ,\beta)$ denote the Hilbert scheme indexed by the class $[\cO_\beta]+\sum_{l,k}m_{l,k}[\cO_{l,k}]\in K(\cZ).$  In \cite{bcy:otv}, orbifold DT invariants are defined via Behrend's constructible function $\nu:Hilb_{\gamma}(\cZ,\beta)\rightarrow \bZ$ \cite{b:dttivmg}.  More precisely, the (multi-regular) DT partition function is defined by
\[
DT_{\cZ,mr}({\bf q},v):=\sum_{\beta,\gamma}e(Hilb_{\gamma}(\cZ,\beta),\nu){\bf q}^\gamma v^\beta
\]
where
\[
e(Hilb_{\gamma}(\cZ,\beta),\nu):=\sum_{k\in\bZ}k\cdot e(\nu^{-1}(k))
\]
with $e(-)$ the topological Euler characteristic.  For our purposes, we will be most interested with the reduced partition function
\[
DT_\cZ({\bf q},v):=\frac{DT_{\cZ,mr}({\bf q},v)}{DT_{\cZ,mr}({\bf q},v=0)}
\]

Notice that $\sum_{k}[\cO_{l,k}]=[\cO_{pt}]$ where $pt$ is a generic (smooth) point on $\cZ$.  For later convenience, we introduce an additional variable $q$ and the relations $\prod_k q_{l,k}=q$ for any $l$.

\subsection{Gromov--Witten/Donaldson--Thomas Correspondence}

Our first result is a GW/DT correspondence in the toric transverse $A$-singularity case.

\begin{Theorem}\label{thm:gwdt}
Let $\cZ$ be a toric Calabi--Yau $3$-orbifold with transverse $A$-singularities. Then, with notation as above,
\[
GW_\cZ^\bullet({\bf x},u,v)=DT_\cZ(\bq,v)
\]
after identifying formal variables by
\[
q\leftrightarrow -e^{\sqrt{-1}u}, \hspace{.5cm} q_{i,j}\leftrightarrow\xi_{n_i}^{-1}e^{-\sum_k\frac{\xi_{n_i}^{-jk}}{n_i}(\xi_{2n_i}^k-\xi_{2n_i}^{-k})x_{i,k}} \hspace{.5cm} (j>0),
\]
\end{Theorem} 

\begin{proof}
Theorem \ref{thm:gwdt} follows immediately from Theorem \ref{thm:vertexcorrespondence} below and Theorem 2.1 in \cite{rz:chialsf}.
\end{proof}

One consequence of Theorem \ref{thm:gwdt} is that the series $GW({\bf x},u,v)$ does not depend on the choice of torus $\bC^*\subset T$ which was required to define the invariants in the non-projective case.

\subsection{Ruan's Crepant Resolution Conjecture}

Let $GW_W^\bullet({\bf T},u)$ be the GW potential of $W$ defined as above, where ${\bf T}$ is a formal paramter in $H^2(W)$. Write ${\bf t}=(t_{A_i},t_{B_i})$ and ${\bf T}=(T_{B_i}, T_{C_i},T_{D_{i,k}})$ where $t_\bullet$ and $T_\bullet$ are formal variables dual to the bases chosen in Section \ref{sec:bases}. Our second result is a proof of the all-genus CRC for $\cZ$ and $W$.

\begin{Theorem}\label{thm:crc}
Let $\pi:W\rightarrow\cZ$ be a crepant resolution of a toric Calabi-Yau $3$-orbifold with transverse $A$-singularities. Then
\[
GW_W^\bullet({\bf T},u)=GW_\cZ^\bullet({\bf x},{\bf t},u)
\]
where the formal parameters are identified by the affine linear transformation
\begin{align*}
T_{B_i}&\rightarrow t_{B_i}\\
T_{C_i}&\rightarrow t_{A_i}-\sum_{l=1}^{n_i-1}(m_i+2)(n_i-l)\left(\frac{2\pi\sqrt{-1}}{n_i}+\sum_{k=1}^{n_i-1}\frac{\xi_{n_i}^{-lk}}{n_i}(\xi_{2n_i}^k-\xi_{2n_i}^{-k})x_{i,k}\right)\\
T_{D_{i,j}}&\rightarrow \frac{2\pi\sqrt{-1}}{n_i}+\sum_{k=1}^{n_i-1}\frac{\xi_{n_i}^{-jk}}{n_i}(\xi_{2n_i}^k-\xi_{2n_i}^{-k})x_{i,k}
\end{align*}
\end{Theorem}

\begin{proof}
The theorem is the composition of
\begin{itemize}
\item the smooth GW/DT correspondence of Maulik--Oblomkov--Okounkov--Pandharipande \cite{moop:gwdtc},
\item the DT CRC \cite{r:dtcrc}, and
\item Theorem \ref{thm:gwdt}.
\end{itemize}
\end{proof}

\section{Local Theory}\label{sec:local}

\subsection{Notation}

In this section, we define the framed $A_{n-1}$ orbifold vertex in both GW and DT theory. For the reader's convenience we provide a table describing our notation. For a more thorough treatment of the notation, see \cite{rz:ggmv,rz:chialsf}.

\begin{tabular}{c || p{8cm}}

$\xi=\xi_n$ & $\exp\left(\frac{2\pi\sqrt{-1}}{n} \right)$\\
$\bZ_n$ & $\left\langle \xi_n \right\rangle$\\
$\tau,\eta,\rho,\omega$ & partitions: $\tau=(\tau_0\geq\tau_1\geq\dots\geq\tau_{l(\tau)-1}\geq 0)$\\
$|\tau|$ & size of the partition: $\sum \tau_i$\\
$\mu,\nu,\lambda,\sigma$ & $n$-tuples of partitions: $\mu=(\mu^0,\dots,\mu^{n-1})$ \\
$|\mu|$ & size of the $n$-tuple of partitions: $\sum |\mu^i|$\\
$\chi_\lambda$ & irreducible representation of $S_{|\lambda|}\wr \bZ_n$ indexed by $\lambda$\\
$\chi_\lambda(\mu)$ & value of the irreducible representation $\lambda$ on the conjugacy class $\mu$\\
$z_\mu$ & order of the centralizer of any element in the conjugacy class $\mu$\\
$\bar\mu$ & $n$-colored Young diagram with $n$-quotient $\mu$\\
$\zeta=(\zeta_1,\dots,\zeta_m)$ & $m$-tuple of elements in $\bZ_n$\\
$\gamma=(\gamma_1,\dots,\gamma_m)$ & $m$-tuple of \emph{nontrivial} elements in $\bZ_n$\\
$\Mbar_{g,m}(\cB\bZ_n)$ & moduli space of genus $g$ $m$-pointed maps to $\cB\bZ_n$\\
$\Mbar_{g,\zeta}(\cB\bZ_n)$ & component of the moduli space mapping the $i$th marked point to the $\zeta_i$ component of the inertia stack $\cI\cB\bZ_n$\\
$\bE_{\xi^i}$ & $i$-th Hodge bundle on $\Mbar_{g,m}(\cB\bZ_n)$\\
$\lambda_j^{\xi^i}$ & $j$th Chern class of $\bE_{\xi^i}$\\
$\psi_i$ & psi-class on $\Mbar_{g,m}(\cB\bZ_n)$ pulled back from $\Mbar_{g,m}$\\
\end{tabular}

\subsection{GW Vertex}

Let $\cX$ be the orbifold $[\bC^3/\bZ_n]$ where $\bZ_n$ acts on the coordinates with weights $(r_1,r_2,r_3)=(1,-1,0)$. Let $\bC^*$ act on $\cX$ with weights $(w_1,w_2,w_3)$ satisfying the CY condition $w_1+w_2+w_3=0$.

Descendant GW invariants of $\cX$ are defined by fixed-point localization with respect to the $\bC^*$ action. Explicitly, we use the following correlator notation.
\begin{align*}
\left\langle \prod_{i=1}^m \zeta_i\psi^{j_i} \right\rangle^{\cX}_{g,m}&:=\int_{\left[\Mbar_{g,\zeta}(\cX)^{\bC^*}\right]^{vir}}\frac{\prod_{i=1}^m\psi^{j_i}}{e^{eq}\left(\cN_{\Mbar^{\bC^*}/\Mbar}\right)}\\
&=\int_{\Mbar_{g,\zeta}(\cB\bZ_n)}\frac{\Lambda^{1}(w_1)\Lambda^{-1}(w_2)\Lambda^0(w_3)}{\delta(w) w_3}\prod_{i=1}^m\psi^{j_i}
\end{align*}
where 
\[
\Lambda^i(t):=(-1)^{\rk\left(\bE_{\xi^i}\right)}\sum_{j=0}^{\rk\left(\bE_{\xi^i}\right)}(-t)^{\rk\left(\bE_{\xi^i}\right)-j}\lambda_j^{\xi^i}
\]
and $\delta(w)$ is the function which takes value $w_1w_2$ on the connected component parametrizing trivial covers of the source curve and takes value $1$ on all other components.

Given ordinary partitions $\tau^+=(\tau^+_i)$, $\tau^-=(\tau^-_i)$ and an $n$-tuple of partitions $\mu=(\mu_i^j)$, we define a $4$-tuple of integers associated to each part $\kappa$ by:
\[
\Delta(\kappa)=(i(\kappa),m(\kappa),k(\kappa),d(\kappa)):=\begin{cases}
(1,n,\tau_i^+ \text{ mod } n, \tau_i^+) &\text{ if }\kappa=\tau_i^+\\
(2,n,-\tau_i^- \text{ mod }n, \tau_i^-) &\text{ if }\kappa=\tau_i^-\\
(3,1,j,\mu_i^j) &\text{ if }\kappa=\mu_i^j\\
\end{cases}
\]

We define the following generating series which is a formal series in $z_i^+$, $z_i^-$, and $z_i^j$
\[
J_{g,\gamma}^\cX(\tau^+,\tau^-,\mu;w):=\left\langle\prod_{i}\gamma_i\prod_i\frac{k(\tau^+_i)}{z_i^+-\psi_i^+}\prod_i\frac{k(\tau^-_i)}{z_i^--\psi_i^-}\prod_{i,j}\frac{k(\mu^j_i)}{z_i^j-\psi_i^j}\right\rangle^\cX_{g,|\gamma|+l}
\]
where
\[
l:=l(\tau^+)+l(\tau^-)+l(\mu).
\]

We define  the (positively oriented) disk function by
\[
D^n(i,m,k,d;w):=\left(\frac{dw_1w_2}{mw_i}\right)^{\delta_{0,k}}\frac{mw_3}{d\left\lfloor\frac{d}{m}\right\rfloor!}\frac{\Gamma\left( \frac{dw_{i+1}}{mw_i}+\left\langle\frac{-kr_{i+2}}{n}\right\rangle+\frac{d}{m} \right)}{\Gamma\left( \frac{dw_{i+1}}{mw_i}-\left\langle\frac{-kr_{i+1}}{n}\right\rangle+1 \right)}
\]
where $w_4:=w_1$, $w_5:=w_2$, and similar for the $r_i$.

The particular Hodge integrals we are interested in take the form
\[
V^\cX_{g,\gamma}(\tau^+,\tau^-,\mu;w):=\frac{J_{g,\gamma}^\cX(\tau^+,\tau^-,\mu;w)\prod_\kappa D^n(\Delta(\kappa);w)}{|\aut(\tau^+)||\aut(\tau^-)||\aut(\mu)|}\Bigg|_{{z_i^+=\frac{nw_1}{\tau_i^+} \atop z_i^-=\frac{nw_2}{\tau_i^-}} \atop z_i^j=\frac{w_3}{\mu_i^j}}
\]

Introduce formal variables $u$ and $x_i$ to track genus and marks.  We define
\begin{equation*}
V^\cX_{\tau^+,\tau^-,\mu}(x,u;w):=\sum_{g,\gamma}V_{g,\gamma}^\cX(\tau^+,\tau^-,\mu;w)u^{2g-2+l}\frac{x^\gamma}{\gamma!}.
\end{equation*}

Also introduce the variables $p_\tau^+$, $p_\tau^-$, $p_\mu:=\prod_j p_{\mu^j}^j$ with formal multiplication defined by concatenating indexing partitions whenever the superscripts agree.  We denote the disconnected vertex by
\begin{equation*}
V_{\tau^+,\tau^-,\mu}^{\cX,\bullet}(x,u;w):=\exp\left( \sum_{(\eta^+,\eta^-,\nu)\neq (\emptyset,\emptyset,\emptyset) }V_{\eta^+,\eta^-,\nu}^\cX(x,u;w)p_{\eta^+}^+p_{\eta^-}^- p_{\nu}\right)\left[ p_{\tau^+}^+p_{\tau^-}^- p_{\mu}\right]
\end{equation*}
where  $[-]$ denotes ``the coefficient of''.  By definition, $V_{\tau^+,\tau^-,\mu}^{\cX,\bullet}(x,u;w)$ is the GW $A_{n-1}$ vertex defined in \cite{r:lgoa}.  For our current purposes, it is more convenient to work with a slight modification.

\begin{definition}\label{framedvertex}
The \emph{framed GW $A_{n-1}$ vertex} is defined by
\begin{align*}
\tilde{V}_{\tau^+,\tau^-,\mu}^{\cX,\bullet}(x,u;w):=&(-1)^{|\mu|+\sum_i\left(\left\lfloor-\frac{\tau^-_i}{n}\right\rfloor\right)}\\
&\cdot\sqrt{-1}^{-l(\tau^+)-l(\tau^-)}\prod_{k=0}^{n-1}(\sqrt{-1}\xi_{2n}^k)^{l_k(\mu)}V_{\tau^+,\tau^-,\mu}^{\cX,\bullet}(x,u;w)
\end{align*}
\end{definition}

\subsection{DT Vertex}

Let $\bq=(q_0,\dots,q_{n-1})$ be formal variables with indices computed modulo $n$.  Define the variables $\fq_i$ recursively by $\fq_0:=1$ and
\[
\fq_t:=q_t\fq_{t-1}
\]
so that $(\dots,\fq_{-2},\fq_{-1},\fq_0,\fq_1,\fq_2,\dots)=(\dots,q_{-1}^{-1}q_{0}^{-1},q_{0}^{-1},1,q_1,q_1q_2,\dots)$.

For a colored Young diagram $\bar\lambda$ corresponding to a $n$-partition $\lambda$ via $n$-quotients, denote the sizes of the rows in $\bar\lambda$ by $(\bar\lambda_0\geq\bar\lambda_1\geq\bar\lambda_2\geq\dots)$.  Define the variables $\fq_{\bullet-\lambda}$ by
\[
\fq_{\bullet-\bar\lambda}:=\{\fq_{-\bar\lambda_0}, \fq_{1-\bar\lambda_1}, \fq_{2-\bar\lambda_2},\dots\}.
\]
In particular, we denote $\fq_{\bullet-\emptyset}=\fq_\bullet$.  An overline on an expression in the $q$ variables denotes interchanging $q_i\leftrightarrow q_{-i}$.

We define
\begin{align*}
P^\cX_{\rho^+,\rho^-,\lambda}:=s_{\bar\lambda}(\bq)\sum_\omega q_0^{-|\omega|}\overline{s_{\rho^+/\omega}(\fq_{\bullet-\bar\lambda})}s_{\rho^-/\omega}(\fq_{\bullet-\bar\lambda'})
\end{align*}
where $s_{\bar\lambda}(\bq)$ denotes the loop Schur function of $\bar\lambda$ in the variables $(q_0,\dots,q_{n-1})$ and $s_{\rho/\omega}$ denotes a skew Schur function.

We modify $P$ to incorporate the framing.

\begin{definition}\label{def:dt}
The \emph{framed DT $A_{n-1}$ vertex} is defined by
\begin{align*}
\tilde P^\cX_{\rho^+,\rho^-,\lambda}(w):&=(-1)^{|\lambda|}q^{\frac{|\lambda|}{2}}\frac{\chi_{\bar\lambda}(n^{|\lambda|})}{\dim(\lambda)}\left( \left(-\xi_{2n}\right)^{-|\lambda|} \prod_k \xi_n^{-k|\lambda^k|} \right)^{\frac{nw_1}{w_3}}\\
&\cdot \left( -q^{\frac{1}{2}}q_1^{-\frac{1}{n}}\cdots q_{n-1}^{-\frac{n-1}{n}}\right)^{|\rho^+|}\left(-q^{\frac{1}{2}}q_1^{-\frac{n-1}{n}}\cdots q_{n-1}^{-\frac{1}{n}} \right)^{|\rho^-|}\\
&\cdot\left(\prod_{(i,j)\in\rho^+}q^{i-j} \right)^{\frac{w_3}{nw_1}}\left(\prod_{(i,j)\in\rho^-}q^{i-j} \right)^{\frac{w_3}{nw_2}}\left(\prod_{(i,j)\in\bar\lambda}q_{j-i}^{i-j} \right)^{\frac{w_1}{w_3}} P^\cX_{\rho^+,\rho^-,\lambda}
\end{align*}
\end{definition}

\section{Vertex GW/DT Correspondence}\label{sec:vertexgwdt}

Stated as Conjecture 2.1 in \cite{rz:chialsf}, the building block for all of our results is the following correspondence.

\begin{Theorem}\label{thm:vertexcorrespondence}
After the identification of variables
\[
q\leftrightarrow e^{\sqrt{-1}u}, \hspace{.5cm} q_j\leftrightarrow\xi_n^{-1}e^{-\sum_k\frac{\xi_n^{-jk}}{n}(\xi_{2n}^k-\xi_{2n}^{-k})x_k} \hspace{.5cm} (j>0),
\]
we have an identification of framed vertex theories:
\[
\tilde V_{\tau^+,\tau^-,\mu}^{\cX,\bullet}(w)=\sum_{\rho^+,\rho^-,\lambda}\tilde P^\cX_{\rho^+,\rho^-,\lambda}(w)\frac{\chi_{\rho^+}(\tau^+)}{z_{\tau^+}}\frac{\chi_{\rho^-}(\tau^-)}{z_{\tau^-}}\frac{\chi_\lambda(\mu)}{z_\mu}.
\]
\end{Theorem}

In \cite{moop:gwdtc}, Maulik--Oblomkov--Okounkov--Pandharipande proved the equivalent of Theorem \ref{thm:vertexcorrespondence} in the smooth case $n=1$.

\begin{remark}
We've omitted the $\alpha$ terms from the definitions in \cite{rz:chialsf}. Incorporating the $\alpha$ terms, it is not hard to show that Theorem \ref{thm:vertexcorrespondence} implies Conjecture 2.1 in \cite{rz:chialsf}.
\end{remark}

The proof of Theorem \ref{thm:vertexcorrespondence} occupies the rest of this section. We proceed with the following steps.
\begin{itemize}
\item We begin with a discussion of the equivariant geometry of $Y$ in order to define the open GW potential of $Y$.
\item We apply the open CRC of Brini--Cavalieri--Ross \cite{bcr:craos} to relate the GW orbifold vertex to the open GW theory of $Y$.
\item We apply the smooth GW/DT correspondence of Maulik--Oblomkov--Okounkov--Pandharipande \cite{moop:gwdtc} to relate the open GW theory of $Y$ to an anologous DT series.
\item We apply the DT CRC of the author \cite{r:dtcrc} which relates the previous DT series to the DT orbifold vertex.
\end{itemize}
Pulling together all of these steps and simplifying proves Theorem \ref{thm:vertexcorrespondence}.

\subsection{Equivariant Geometry of $Y$}

Let $Y$ be the toric resolution of $\cX$. Then $Y$ contains a chain of $n-1$ $\bP^1$s, all with normal bundle $\cO\oplus\cO(-2)$. In Figure \ref{fig:webdiagrams}, we've depicted the web diagrams for $\cX$ and $Y$. We equip $Y$ with a $\bC^*$ action compatible with that on $\cX$, the labeling in Figure \ref{fig:webdiagrams} indicates the weights of this action.

\begin{figure}
\begin{tikzpicture}
\begin{scope}[scale=1.5,xshift=5cm]
\draw (0,0) -- (1, .75) -- (2,2) -- (2.1,2.4);
\draw (2.1,3.1) -- (2,3.5) -- (1,4.75) -- (0,5.5);
\draw (1,.75) -- (3.5,.75);
\draw (2,2) -- (3.5,2);
\draw (2,3.5) -- (3.5,3.5);
\draw (1,4.75) -- (3.5,4.75);
\draw (2.1,2.8) node{$\vdots$};
\draw (3,1) node{$w_3$};
\draw (3,2.25) node{$w_3$};
\draw (3,3.75) node{$w_3$};
\draw (3,5) node{$w_3$};
\draw (.1,5) node{$nw_1$};
\draw (.4,4.5) node{$-nw_1-w_3$};
\draw (1.1,3.75) node{$nw_1+w_3$};
\draw (1.2,3.25) node{$-nw_1-2w_3$};
\draw (.1,.5) node{$nw_2$};
\draw (.4,1) node{$-nw_2-w_3$};
\draw (1.1,1.75) node{$nw_2+w_3$};
\draw (1.2,2.25) node{$-nw_2-2w_3$};
\end{scope}
\begin{scope}[scale=1.5,xshift=.5cm]
\draw (1,1.75) -- (2,2.75) -- (1,3.75);
\draw (3,3) node{$w_3$};
\draw (1,3.4) node{$w_1$};
\draw (1,2.1) node{$w_2$};
\end{scope}
\begin{scope}[scale=1.5, very thick ,xshift=.5cm]
\draw (2,2.75) -- (3.5,2.75);
\end{scope}
\end{tikzpicture}
\caption{}\label{fig:webdiagrams}
\end{figure}

Let $D_1,\dots,D_{n-1}\in H_{\bC^*}^*(Y,\bQ)$ denote the equivariant cohomology classes corresponding to (the canonical equivariant lifts of) the divisors in $Y$ which are dual to the $n-1$ $\bP^1$s, labeled from bottom to top with respect to Figure \ref{fig:webdiagrams}. Via the Atiyah-Bott localization isomorphism, let $P_0,\dots,P_{n-1}\in H_{\bC^*}^*(Y,\bQ)$ denote the equivariant cohomology classes corresponding to the fixed points of $Y$.

\subsection{Open GW Theory of $Y$}

As in the case for the orbifold, descendant GW invariants of $Y$ are defined via fixed point localization. Explicitly, for divisor classes $\zeta_1,\dots,\zeta_m\in \{D_1,\dots,D_{n-1}\}$ we define
\[
\left\langle \prod_i\zeta_i\psi^{j_i} \right\rangle_{g,m,\beta}^Y:=\sum_{F\subset\Mbar_{g,m}(Y,\beta)^{\bC^*}} \int_{[F]^{vir}}\frac{ev_i^*(\zeta_i)\psi^{j_i}}{e^{eq}\left(\cN_{F}\right)}
\]

Analogous to the case of $\cX$, we define the following formal series in $z_i^+$, $z_i^+$, and $z_i^j$
\begin{equation*}
J^Y_{g,\zeta,\beta}(\tau^+,\tau^-,\mu;w):=\left\langle \prod_i\zeta_i \prod_i\frac{P_{n-1}}{z_i^+-\psi_i^+}\prod_i\frac{P_{0}}{z_i^--\psi_i^i}\prod_{i,j}\frac{P_{j}}{z_i^j-\psi_i^j} \right\rangle_{g,|\zeta|+l,\beta}^Y
\end{equation*}

Define $w^j:=(-nw_2-(j+1)w_3,nw_2+jw_3,w_3)$ and set $w^+:=w^{n-1}$ and $w^-:=w_0$. For each part $\kappa$ in $\tau^+$, $\tau^-$, or $\mu$, define
\[
\tilde\Delta(\kappa)=(\tilde i(\kappa),\tilde m(\kappa),\tilde k(\kappa),\tilde d(\kappa); w(\kappa)):=\begin{cases}
(1,1,0, \tau_i^+;w^+) &\text{ if }\kappa=\tau_i^+\\
(2,1,0, \tau_i^-;w^-) &\text{ if }\kappa=\tau_i^-\\
(3,1,0,\mu_i^j;w^j) &\text{ if }\kappa=\mu_i^j\\
\end{cases}
\]

Then the open GW invariants of $Y$ are defined by

\[
V^Y_{g,\zeta,\beta}(\tau^+,\tau^-,\mu;w):=\frac{J_{g,\zeta,\beta}^Y(\tau^+,\tau^-,\mu;w)\prod_\kappa D^1(\tilde\Delta(\kappa))}{|\aut(\tau^+)||\aut(\tau^-)||\aut(\mu)|}\Bigg|_{{z_i^+=\frac{nw_1}{\tau_i^+} \atop z_i^-=\frac{nw_2}{\tau_i^-}} \atop z_i^j=\frac{w_3}{\mu_i^j}}
\]

Introduce formal variables $t_i$ to track divisor insertions.  We define
\begin{equation*}
V^Y_{\tau^+,\tau^-,\mu}(t,u;w)_\beta:=\sum_{g,\zeta}V_{g,\zeta,\beta}^Y(\tau^+,\tau^-,\mu;w)u^{2g-2+l}\frac{t^\zeta}{\zeta!}
\end{equation*}
and
\[
V^Y_{\tau^+,\tau^-,\mu}(t,u;w):=\sum_\beta V^Y_{\tau^+,\tau^-,\mu}(t,u;w)_\beta
\]

We denote the disconnected series by
\begin{equation*}
V_{\tau^+,\tau^-,\mu}^{Y,\bullet}(t,u;w):=\exp\left( \sum_{(\eta^+,\eta^-,\nu)\neq (\emptyset,\emptyset,\emptyset) }V_{\eta^+,\eta^-,\nu}^Y(t,u;w)p_{\eta^+}^+p_{\eta^-}^- p_{\nu} \right)\left[ p_{\tau^+}^+p_{\tau^-}^- p_{\mu}\right]
\end{equation*}
and we define the modified series 
\[
\tilde V_{\tau^+,\tau^-,\mu}^{Y,\bullet}(t,u;w):=(-1)^{|\tau^-|}\sqrt{-1}^{l(\mu)-l(\tau^+)-l(\tau^-)}V_{\tau^+,\tau^-,\mu}^{Y,\bullet}(t,u;w)
\]


\subsection{The Open CRC and Divisor Equation}

Carefully unpackaging the main results from \cite{bcr:craos}, we have the following correspondence.

\begin{theorem}[\cite{bcr:craos}, Theorems 4.2 and 4.4]\label{thm:ocrc}
After the identification of variables
\[
t_j \leftrightarrow  -\frac{2\pi\sqrt{-1}}{n}-\sum_k\frac{\xi_n^{-jk}}{n}(\xi_{2n}^k-\xi_{2n}^{-k})x_k
\]
we have an identification of formal series:
\[
\tilde{V}_{\tau^+,\tau^-,\mu}^{\cX,\bullet}(w)=\sum_{\nu}\frac{\delta_\mu(\nu)}{n^{l(\nu)}} \left(\prod_k(-1)^{(n-k-1)|\nu_k|}(\xi_{2n}^{-1}\xi_n^{-k})^{|\nu_k|\frac{nw_1}{w_3}}\right)\frac{\tilde V_{\tau^+,\tau^-,\nu}^{Y,\bullet}(w)}{\tilde V_{\emptyset,\emptyset,\emptyset}^{Y,\bullet}(w)}
\]
where
\[
\delta_{\mu}(\nu):=\frac{|\aut(\nu)|}{|\aut(\mu)|}\prod_{j=0}^{n-1}\prod_{i=0}^{l_j(\mu)}\left( \sum_{k=0}^{n-1} \xi_{n}^{-jk}p_{\mu_i^j}^k \right)[p_\nu]
\]
\end{theorem}

\begin{remark}
Notice that $\delta_\mu(\nu)=0$ whenever $\underline{\mu}\neq\underline{\nu}$. In the language of \cite{bcr:craos}, the sum over $\nu$ is simply a way of distributing disks of winding $\mu_i^j$ over the $n$ horizontal legs of the resolution, each with the appropriate factor dictated by Theorem 4.2 of \cite{bcr:craos}.
\end{remark}

\begin{remark}
We need to divide by $\tilde V_{\emptyset,\emptyset,\emptyset}^{Y,\bullet}(w)$ on the right hand side because this corresponds, by a theorem of Zhou \cite{z:crcag}, to the closed GW theory of $\cX$ which we've removed by hand in the definitions.
\end{remark}

We can use the Atiyah-Bott localization isomorphism to write $D_i$ in terms of the idempotent fixed point basis:
\[
D_i\rightarrow -\sum_{j<i}(n-i)w_2P_j-\sum_{j\geq i}jw_1P_j
\]

This allows us to apply the divisor equation to extract the $t$-dependence of our generating series:
\begin{align*}
&\tilde V_{\tau^+,\tau^-,\nu}^{Y,\bullet}(t,u;w)_\beta=\tilde V_{\tau^+,\tau^-,\nu}^{Y,\bullet}(0,u;w)_\beta\\
&\cdot\exp\left(\sum t_i\left(\beta_i -\frac{i}{n}|\tau^+|-\frac{n-i}{n}|\tau^-|-\sum_{j<i}(n-i)\frac{w_2}{w_3}|\nu^j|-\sum_{j\geq i}i\frac{w_1}{w_3}|\nu^j| \right) \right)
\end{align*}
where $\beta_i:=\int_\beta D_i$.

If we define the variables $v_i:=e^{t_i}$, then we have
\begin{align}\label{eqn:divisor}
&\tilde V_{\tau^+,\tau^-,\nu}^{Y,\bullet}(t,u;w)=\tilde V_{\tau^+,\tau^-,\nu}^{Y,\bullet}(u;v;w)\\
&\nonumber\cdot\left(v_1^{-\frac{1}{n}}\cdots v_{n-1}^{-\frac{n-1}{n}}  \right)^{|\tau^+|}\left(v_1^{-\frac{n-1}{n}}\cdots v_{n-1}^{-\frac{1}{n}}  \right)^{|\tau^-|}\prod_j\left(\prod_{i\leq j}v_i^{i\frac{w_1}{w_3}}\prod_{i>j}v_i^{(n-i)\frac{w_2}{w_3}}  \right)^{|\nu_j|}
\end{align}
where
\[
\tilde V_{\tau^+,\tau^-,\nu}^{Y,\bullet}(u;v;w):=\sum_\beta \tilde V_{\tau^+,\tau^-,\nu}^{Y,\bullet}(0,u;w)_\beta v^\beta
\]

Theorem \ref{thm:ocrc} along with equation \eqref{eqn:divisor} provide us with an expression of $\tilde{V}_{\tau^+,\tau^-,\mu}^{\cX,\bullet}(x,u;w)$ in terms of $\tilde V_{\tau^+,\tau^-,\nu}^{Y,\bullet}(u;v;w)$ where formal parameters are identified by 
\[
v_j\leftrightarrow\xi_n^{-1}e^{-\sum_k\frac{\xi_n^{-jk}}{n}(\xi_{2n}^k-\xi_{2n}^{-k})x_k}
\]

\subsection{Smooth GW/DT Correspondence}

By the gluing algorithm of Diaconescu--Florea \cite{df:lgta} (see also \cite{r:lgoa}), we have the following.
\begin{theorem}[\cite{df:lgta}, Section 4] \label{thm:gwglue}
\[
\tilde V_{\tau^+,\tau^-,\nu}^{Y,\bullet}(u;v;w)=\sum_{\tau^i}\prod_{j=1}^{n-1}v_j^{|\tau_j|}z_{\tau^j}\prod_{j=0}^n \tilde V^{\bC^3,\bullet}_{\tau^{j+1},\tau^{j},\nu^{j}}(u;w^j)
\]
where $\tau^0:=\tau^-$ and $\tau^n:=\tau^+$. 
\end{theorem}

Moreover, the vertex correspondence of Maulik--Oblomkov--Okounkov--Pandharipande \cite{moop:gwdtc} can be rephrased in our notation as the following.
\begin{theorem}[\cite{moop:gwdtc}, Proposition 2]\label{thm:sgwdt}
After the change of variables $q=e^{\sqrt{-1}u}$,
\[
\tilde V^{\bC^3,\bullet}_{\tau^{j+1},\tau^{j},\nu^{j}}(u;w^j)=\sum_{\rho^{j+1},\rho^j,\lambda^j}\tilde P^{\bC^3}_{\rho^{j+1},\rho^j,\lambda^j}(q;w^j)\frac{\chi_{\rho^{j+1}}(\tau^{j+1})}{z_{\tau^{j+1}}}\frac{\chi_{\rho^j}(\tau^j)}{z_{\tau^j}}\frac{\chi_{\lambda^j}(\nu^j)}{z_{\nu^j}}.
\]
\end{theorem}

If we define
\[
\tilde P_{\rho^+,\rho^-,\lambda}^{Y}(q;v;w)=\sum_{\rho^i}\prod_{j=1}^{n-1}v_j^{|\rho_j|}\prod_{j=0}^n \tilde P^{\bC^3}_{\rho^{j+1},\rho^{j},\lambda^{j}}(q;w^j)
\]
with $\rho^0:=\rho^-$ and $\rho^n:=\rho^+$, then Theorems \ref{thm:gwglue} and \ref{thm:sgwdt} imply that
\begin{equation}\label{eqn:gwdt}
\tilde V_{\tau^+,\tau^-,\nu}^{Y,\bullet}(u;v;w)=\sum_{\rho^+,\rho^-,\lambda}\tilde P^{Y}_{\rho^+,\rho^-,\lambda}(q;v;w)\frac{\chi_{\rho^+}(\tau^+)}{z_{\tau^+}}\frac{\chi_{\rho^-}(\tau^-)}{z_{\tau^-}}\prod_j \frac{\chi_{\lambda^j}(\nu^j)}{z_{\nu^j}}
\end{equation}

\subsection{DT CRC}

The vertex DT CRC can be rephrased in our notation as follows.

\begin{theorem}[\cite{r:dtcrc}, Theorem 3.1]\label{thm:dtcrc}
After the change of variables $v_i\leftrightarrow q_i$, 
\begin{align*}
&\frac{\tilde P^{Y}_{\rho^+,\rho^-,\lambda}(q;v;w)}{\tilde P^{Y}_{\emptyset,\emptyset,\emptyset}(q;v;w)}=\tilde P^{\cX}_{\rho^+,\rho^-,\lambda}(\bq;w)\\
&\cdot\frac{\prod_k\left(\xi_{2n}\xi_n^{k}\right)^{|\lambda^k|\frac{w_1}{w_3}}\prod_{(i,j)\in\lambda^k}(-1)^{n-k-1}q^{(i-j)\frac{nw_1}{w_3}}\prod_{l>k}q^{l-n}}{\left(q_1^{-\frac{1}{n}}\cdots q_{n-1}^{-\frac{n-1}{n}}\right)^{|\rho^+|}\left(q_1^{-\frac{n-1}{n}}\cdots q_{n-1}^{-\frac{1}{n}}\right)^{|\rho^-|}\left(\prod_{(i,j)\in\bar\lambda}q_{j-i}^{i-j}\right)^{\frac{w_1}{w_3}}}
\end{align*}
\end{theorem}

\subsection{Final Computations}

In order to prove Theorem \ref{thm:vertexcorrespondence}, we combine Theorem \ref{thm:ocrc}, equations \eqref{eqn:divisor} and \eqref{eqn:gwdt}, and Theorem \ref{thm:dtcrc} in order to write $\tilde{V}_{\tau^+,\tau^-,\mu}^{\cX,\bullet}(w)$ in terms of $\tilde V_{\tau^+,\tau^-,\nu}^{Y,\bullet}(w)$ after the change of variables. Explicitly, after making the obvious cancellations, we have:
\begin{align*}
\tilde{V}_{\tau^+,\tau^-,\mu}^{\cX,\bullet}(w)&=\sum_{\nu}\delta_\mu(\nu) \sum_{\rho^+,\rho^-,\lambda}\tilde P^{\cX}_{\rho^+,\rho^-,\lambda}(w)\frac{\chi_{\rho^+}(\tau^+)}{z_{\tau^+}}\frac{\chi_{\rho^-}(\tau^-)}{z_{\tau^-}}\prod_j \frac{\chi_{\lambda^j}(\nu^j)}{z_{\nu^j}}\\
&\cdot\left( \left(\prod_{(i,j)\in\lambda^k}q^{n(i-j)}\prod_{i\leq k}v_i^{i}\prod_{i>k}v_i^{i-n}  \right)\prod_{(i,j)\in\bar\lambda}q_{j-i}^{j-i}\right)^{\frac{w_1}{w_3}}
\end{align*}

Theorem \ref{thm:vertexcorrespondence} follows now from the following facts.

\begin{lemma}
With notation as above,
\begin{enumerate}
\item 
\[
\left(\prod_{(i,j)\in\lambda^k}q^{n(i-j)}\prod_{i\leq k}v_i^{i}\prod_{i>k}v_i^{i-n}  \right)\prod_{(i,j)\in\bar\lambda}q_{j-i}^{j-i}=1
\]
\item 
\[
\sum_\nu \delta_\mu(\nu)\sum_\lambda \prod_j\frac{\chi_{\lambda^j}(\nu^j)}{z_{\nu^j}}=\frac{\chi_\lambda(\mu)}{z_\mu}
\]
\end{enumerate}
\end{lemma}

\begin{proof}
The first identity is an exercise in $n$-quotients. In particular, it follows easily from the discussion in Section 3.2 of \cite{r:dtcrc}.

The second identity follows from the infinite wedge expression for the characters of the generalized symmetric group. In particular, it is easily proved from equation (6--3) in \cite{rz:ggmv}.
\end{proof}

\bibliographystyle{alpha}

\end{document}